\documentclass[reqno,11pt]{amsart}
\usepackage{amsmath,amssymb}
\usepackage{graphicx}
\usepackage[top=1in,bottom=1in,left=1in,right=1in]{geometry}
\usepackage[colorlinks,citecolor=red,linkcolor=blue]{hyperref}

\newtheorem{theorem}{Theorem}[section]
\newtheorem{lemma}[theorem]{Lemma}

\newtheorem{proposition}[theorem]{Proposition}
\theoremstyle{definition}

\theoremstyle{remark}
\newtheorem{remark}[theorem]{Remark}

\numberwithin{equation}{section}

\begin{document}

\title[Critical behavior  for a semilinear parabolic equation...]
{Critical behavior  for a semilinear parabolic equation with  forcing term  depending of time and space}

\author[M. Jleli, T. Kawakami, B. Samet]{Mohamed Jleli, Tatsuki Kawakami, Bessem Samet}



\begin{abstract}
We investigate  the large-time behavior of the sign-changing solution of 
the inhomogeneous semilinear heat equation $\partial_t u=\Delta u +|u|^p +t^\sigma w(x)$ 
in $(0,T)\times \mathbb{R}^N$, where $N\ge 2$, $p>1$, $\sigma>-1$, $\sigma\ne 0$ and $w\not\equiv 0$.
The novelty of this paper lies in considering a forcing term ($t^\sigma w(x)$) which depends  both of time and space.
We show that there is an exponent $p^*(\sigma)$ which is critical in the following sense: 
the solution of the above problem blows up in finite time when $1<p<p^*(\sigma)$ 
and $\int_{\mathbb{R}^N} w(x)\,dx>0$, 
while global solution exists for suitably small initial data and $w$ belonging to certain Lebesgue spaces 
when $p\geq p^*(\sigma)$.   
Our obtained results show that 
the forcing term induces an interesting phenomenon of discontinuity of the critical exponent $p^*(\sigma)$. 
Namely, we found that $\displaystyle\lim_{\sigma\to 0^-}p^*(\sigma)\ne  \displaystyle \lim_{\sigma\to 0^+}p^*(\sigma)$.
Furthermore, $\displaystyle\lim_{\sigma\to 0^-}p^*(\sigma)$ coincides 
with the critical exponent of the above problem with $\sigma=0$. 
\end{abstract}

\maketitle

\section{Introduction}

In this paper we investigate  the global existence and blow-up of sign-changing solutions of 
the following inhomogeneous parabolic equation

\begin{equation}
\label{eq:1.1}
\left\{
\begin{array}{ll}
\displaystyle{\partial_t u =\Delta u + |u|^p + t^\sigma w(x)} 
&\mbox{in} \quad (0,T)\times \mathbb{R}^N,
\vspace{5pt}
\\
u(0,x)= u_0(x) 
&\mbox{in} \quad\mathbb{R}^N,
\end{array}
\right.
\end{equation}
where $N\ge 2$, $p>1$, $\sigma>-1$, $\sigma\ne 0$ and $w\not\equiv 0$. 
Namely,  we identify the critical exponent for problem \eqref{eq:1.1},
which separates the nonexistence/existence of global-in-time solutions of \eqref{eq:1.1},
and show the discontinuity of this critical exponent at $\sigma=0$.

In the case $w\equiv 0$ with a nonnegative initial data, problem \eqref{eq:1.1} reduces to 
\begin{equation}
\label{eq:1.1w=0}
\left\{
\begin{array}{ll}
\displaystyle{\partial_t u =\Delta u + u^p } 
&\mbox{in} \quad (0,T)\times \mathbb{R}^N,
\vspace{5pt}
\\
u(0,x)= u_0(x)\ge0
&\mbox{in} \quad\mathbb{R}^N.
\end{array}
\right.
\end{equation}
Fujita \cite{Fujita} established the following results for problem \eqref{eq:1.1w=0}:
\begin{itemize}
\item[(I)] If $1<p<1+2/N$,  then \eqref{eq:1.1w=0}  admits no nontrivial global-in-time solutions.
\item[(II)] If $p>1+2/N$, then  \eqref{eq:1.1w=0} possesses global-in-time solutions for some small $u_0$.
\end{itemize}
Later, it was shown that the borderline case $p=1+2/N$ belongs to the blow-up category
(see e.g. \cite{AW,HA,KO,S,W}).
From above results, the number
\begin{equation}
\label{Fujexp}
p_F:=1+\frac{2}{N}
\end{equation}
is called  the critical Fujita exponent,
which separates the nonexistence/existence of global-in-time solutions of \eqref{eq:1.1w=0}.
In \cite{W}, Weissler also proved that, for the case $p>p_F$, if $\|u_0\|_{L^d}$ is sufficiently small with
$$
d=\frac{N(p-1)}{2}>1,
$$
then  \eqref{eq:1.1w=0} has global positive solutions.

In the case $\sigma=0$, problem \eqref{eq:1.1} reduces to  
\begin{equation}
\label{eq:1.1sigma0}
\left\{
\begin{array}{ll}
\displaystyle{\partial_t u =\Delta u + |u|^p + w(x)} 
&\mbox{in} \quad (0,T)\times \mathbb{R}^N,
\vspace{5pt}
\\
u(0,x)= u_0(x) 
&\mbox{in} \quad\mathbb{R}^N.
\end{array}
\right.
\end{equation}
Problem \eqref{eq:1.1sigma0} was investigated by Bandle et al. \cite{BLZ}. Namely, it was shown that  
\begin{itemize}
\item[(I)] If $1<p<p^*$ and $\int_{\mathbb{R}^N} w(x)\,dx>0$, where
\begin{equation}\label{p*}
p^*=\left\{\begin{array}{lll}
\infty &\mbox{if}& N=1,2, 
\vspace{5pt}
\\
\frac{N}{N-2} &\mbox{if}& N\geq 3,
\end{array}
\right.
\end{equation}
then \eqref{eq:1.1sigma0} has no global solutions.
\item[(II)] If $N\geq 3$ and $p>p^*$, then for any $\delta>0$, 
there exists $\epsilon>0$ such that  \eqref{eq:1.1sigma0} has global solutions provided that 
$$
\max\{|w(x)|, |u_0(x)|\}\le \frac{\epsilon}{\left(1+|x|^{N+\delta}\right)}
$$
regardless of whether or not  $\int_{\mathbb{R}^N} w(x)\,dx>0$. 
\item[(III)] If $N\geq 3$, $p=p^*$,  $\int_{\mathbb{R}^N} w(x)\,dx>0$, 
$w(x)=O(|x|^{-\epsilon-N})$ as $|x|\to \infty$ for some $\epsilon>0$, and either $u\geq 0$ or 
$$
\int_{|x|>R} \frac{w^-(y)}{|x-y|^{N-2}}\,dy =o(|x|^{-N+2})
$$
when $R$ is large, then \eqref{eq:1.1sigma0} has no global solutions. Here, $w^-=\max\{-w,0\}$.   
\end{itemize}

In \cite{Zhang}, Zhang investigated the initial  value problem 
\begin{equation}
\label{eq:1.1sigma0Z}
\left\{
\begin{array}{ll}
\displaystyle{\partial_t u =\Delta u + u^p + w(x)} 
&\mbox{in} \quad (0,\infty)\times M^N,
\vspace{5pt}
\\
u(0,x)= u_0(x) 
&\mbox{in} \quad M^N,
\end{array}
\right.
\end{equation}
where $N\geq 3$, $M^N$ is a non-compact complete Riemannian manifold, 
$\Delta$ is the Laplace-Beltrami operator, 
$u_0\geq 0$ and $w\geq 0$ is a nontrivial $L^1_{loc}$ function.
He proved that 
$$
p_M=\frac{\alpha}{\alpha-2},
$$
where $\alpha>2$ is the decay rate of the fundamental solution of  $\partial_t u=\Delta u$ in $M^N$, 
is the critical Fujita exponent for problem \eqref{eq:1.1sigma0Z}.
(See also \cite{Pinsky}.)
Moreover, it was shown that if the Ricci curvature of $M^N$ is non-negative, 
then $p_M$ belongs to the blow-up case. 
Note that in the case $M^N=\mathbb{R}^N$ ($N\ge 3$), one has $p_M=p^*$, where $p^*$ is given by \eqref{p*}. 
On the other hand, observe that $p^*>p_F$, where $p_F$ is the critical Fujita exponent of \eqref{eq:1.1w=0}
given by \eqref{Fujexp}. 
This means that the additional forcing term $w=w(x)\geq 0$, no matter how small it is, 
has the effect of increasing the critical exponent.  
A similar phenomenon was observed recently for a nonlocal-in-time nonlinear heat equation \cite{JS}.


In all the above cited works, the considered inhomogeneous term depends only of space ($w=w(x)$).  
In this paper 
we investigate, for the first time, the parabolic equation \eqref{eq:1.1} with the forcing term $t^\sigma w(x)$.  
We show that there is an exponent $p^*(\sigma)$ which is critical in the following sense: 
when $1<p<p^*(\sigma)$ and $\int_{\mathbb{R}^N} w(x)\,dx>0$, 
the solution of problem \eqref{eq:1.1} blows up in finite time; 
when $p\ge p^*(\sigma)$, the solution is global for suitably small $u_0$ and $w$.  

As usual, \eqref{eq:1.1} is equivalent in the appropriate setting to
\begin{equation}
\label{eq:1.2}
u(t)=e^{t\Delta} u_0+\int_0^t e^{(t-s)\Delta} \left(|u(s)|^p+s^\sigma w\right)\,ds,\quad 0\le t\le  T,
\end{equation}
where $e^{t\Delta}$ is the heat semigroup on $\mathbb{R}^N$. 
Namely, for $u_0\in C_0(\mathbb R^N)$ and $w\in C_0^\alpha(\mathbb R^N)$ with $\alpha\in(0,1)$, 
one can  see that the solution $u$ of the integral equation \eqref{eq:1.2} satisfies \eqref{eq:1.1} 
in the classical sense (see Proposition~\ref{Proposition:2.1}.). 

Our obtained results are given by the following theorems.  
We discuss separately the cases  $-1<\sigma<0$ and $\sigma>0$.

\begin{theorem}\label{Theorem:1.1} 
Let $N\ge2$, $p>1$ and $\sigma\in(-1,0)$.
Assume $w\in C_0^\alpha(\mathbb{R}^N)\cap L^1(\mathbb{R}^N)$ for some $\alpha\in(0,1)$.
Then the following holds.
\begin{itemize}
\item[(i)] 
Assume
\begin{equation}
\label{eq:1.3}
1<p<\frac{N-2\sigma}{N-2-2\sigma}
\end{equation}
and $\int_{\mathbb{R}^N} w(x)\,dx>0$.
Then for any $u_0\in C_0(\mathbb{R}^N)$,
the solution of \eqref{eq:1.2} blows up in finite time.
\item[(ii)] 
Assume
\begin{equation}
\label{eq:1.4}
p\ge \frac{N-2\sigma}{N-2-2\sigma}.
\end{equation}
Put
\begin{equation}
\label{eq:1.5}
d=\frac{N(p-1)}{2}, \quad  
k=\frac{d}{p(\sigma+1)-\sigma}.
\end{equation}
Then for any $u_0\in C_0(\mathbb R^N)\cap L^d(\mathbb{R}^N)$ and  $w$ with
$\|u_0\|_{L^d(\mathbb{R}^N)}+\|w\|_{L^k(\mathbb{R}^N)}$ is sufficiently small, 
the solution $u$ of \eqref{eq:1.2} exists globally. 
\end{itemize}
\end{theorem}

\begin{theorem}\label{Theorem:1.2}
Let $N\ge2$, $p>1$ and $\sigma>0$.
Assume $w\in C_0^\alpha(\mathbb{R}^N)\cap L^1(\mathbb{R}^N)$ for some $\alpha\in(0,1)$ and $\int_{\mathbb{R}^N} w(x)\,dx>0$.
Then for any $u_0\in C_0(\mathbb{R}^N)$, 
the solution of \eqref{eq:1.2} blows up in finite time.
\end{theorem}

\begin{remark}\label{Remark:1.1}
{\rm(i)} No assumption on the sign of $u_0$  is needed in Theorems {\rm\ref{Theorem:1.1}} 
and {\rm\ref{Theorem:1.2}}.
\vspace{3pt}
\newline
{\rm(ii)} From Theorems {\rm\ref{Theorem:1.1}} and {\rm\ref{Theorem:1.2}},  
one observes that the critical exponent for \eqref{eq:1.1} is given by
$$
p^*(\sigma):=
\left\{
\begin{array}{lll}
{\displaystyle\frac{N-2\sigma}{N-2-2\sigma}} 
&\mbox{if}
& -1<\sigma<0,\vspace{5pt}\\
\infty &\mbox{if}& \sigma>0.
\end{array}
\right.
$$
Observe also that 
when $N\geq 3$, $\displaystyle\lim_{\sigma\to 0^-}p^*(\sigma)\ne \displaystyle\lim_{\sigma\to 0^+}p^*(\sigma)$.
\vspace{3pt}
\newline
{\rm(iii)} Observe that $\displaystyle\lim_{\sigma\to 0^-}p^*(\sigma)=p^*$ (which is  given by \eqref{p*}) 
is the critical exponent for problem \eqref{eq:1.1sigma0}
and also the critical exponent for problem \eqref{eq:1.1sigma0Z} in the case $M^N=\mathbb{R}^N$, $N\geq 3$.
\vspace{3pt}
\newline
{\rm(iv)} In the assertion~{\rm(ii)} of Theorem~{\rm\ref{Theorem:1.1}}, 
one can relax the smallness assumptions for initial data $u_0$ and the inhomogeneous term $w(x)$
from the Lebesgue space $L^r$ to the Lorentz space $L^{r,\infty}$ (the weak $L^r$ space). 
In fact, applying the same argument as in the proof of the assertion~{\rm(ii)} of Theorem~{\rm\ref{Theorem:1.1}}
with the weak Young inequality (see e.g. \cite[(G2)]{FKS}), 
one can get the same conclusion for the case $p>p^*(\sigma)$.
Then we can consider 
$$
|u_0(x)|\sim |x|^{-\frac{N}{d}},\qquad
|w(x)|\sim |x|^{-\frac{N}{k}}
$$
for sufficiently large $x$, which do not belong to $L^d(\mathbb R^N)$ and $L^k(\mathbb R^N)$, respectively. Furthermore, for the critical case $p=p^*(\sigma)$, namely $k=1$, by \eqref{eq:4.3},
one can only relax the smallness assumption for initial data $u_0$.
Therefore it is still open that, for $w$ which behaves like $|x|^{-N}$ for sufficiently large $x$,  
there exists a global-in-time solution of \eqref{eq:1.2} or not.
\end{remark}

The rest of the paper is organized as follows. 
In Section \ref{section:2}, we investigate the local existence properties for equation \eqref{eq:1.1}. 
The  assertion~(i) of Theorem \ref{Theorem:1.1},
as well as Theorem \ref{Theorem:1.2} are established in Section \ref{section:3}. 
The next section is devoted to the proof of the global existence result given by 
the assertion~(ii) of Theorem \ref{Theorem:1.1}.

\section{Local existence}\label{section:2}

We first introduce some notations. 
For any $1\le r\le\infty$, we denote by $\|\cdot\|_{L^r}$ the usual norm of $L^r:=L^r(\mathbb{R}^N)$.  
Let $C_0(\mathbb R^N)$ be the space of continuous functions in $\mathbb R^N$ vanishing at infinity.
For $\alpha\in(0,1)$, let $C^\alpha_0(\mathbb R^N)=C^\alpha(\mathbb R^N)\cap C_0(\mathbb R^N)$.  
By the letter $C$,  we denote generic positive constants and they may have different values also within the same line. 
\vspace{3pt}

Further, let us recall some well known facts about the semigroup $e^{t\Delta}$. 
There exists a positive  constant $c_1$ such that for any $1\le q\le r\le \infty$, one has
\begin{equation}\label{eq:2.1}
\|e^{t\Delta}\varphi\|_{L^r}\le c_1t^{-\frac{N}{2}(\frac{1}{q}-\frac{1}{r})}\|\varphi\|_{L^q},\qquad t>0,
\end{equation}
for any $\varphi\in L^q$. In particular, 
\begin{equation}\label{eq:2.2}
\|e^{t\Delta}\varphi\|_{L^q}\le \|\varphi\|_{L^q},\quad t>0.
\end{equation}
Furthermore, for $\varphi\in C_0(\mathbb R^N)$, it holds that (see e.g. \cite{GGS})
$$
\lim_{t\to0^+}e^{t\Delta}\varphi(x)=\varphi(x),\qquad x\in{\mathbb R^N}.
$$
Applying these estimates, we prove the following local existence result. 

\begin{proposition}\label{Proposition:2.1}
Let $N\ge2$, $p>1$, $\sigma>-1$ with $\sigma\neq 0$. 
Assume $u_0\in C_0(\mathbb{R}^N)$ and $w\in C^\alpha_0(\mathbb R^N)$ for some $\alpha\in(0,1)$.
Then the following holds.
\begin{itemize}
\item[(i)] 
There exists $0<T<\infty$ and a unique solution  $u\in C([0,T],C_0(\mathbb{R}^N))$ of \eqref{eq:1.2}.
Furthermore, the solution $u$ satisfies \eqref{eq:1.1} in the classical sense.
\vspace{3pt}
\item[(ii)] 
The solution $u$ can be extended to a maximal interval $[0,T_{\max})$, where $0<T_{\max}\le \infty$,
and if $T_{\max}<\infty$, then $\lim_{t\to T_{\max}^-}\|u(t)\|_{L^{\infty}}= \infty$.
\vspace{3pt}
\item[(iii)] 
If, in addition,  $u_0,w\in L^r$, where $1\le r<\infty$, then 
$u\in C([0,T_{\max}),C_0(\mathbb{R}^N))\cap C([0,T_{\max}),L^r)$.
\end{itemize}
\end{proposition}

\begin{proof}
The proof of this proposition follows from  standard arguments.
For the completeness of this paper, we write the details.
\vspace{3pt}

We first prove the assertion~(i). For the uniqueness of solutions,
let $T>0$ and $u,v\in C([0,T],C_0(\mathbb{R}^N))$ be two solutions of \eqref{eq:1.2}. Since it holds that
\begin{equation}
\label{eq:2.3}
|a^p-b^p|\le p\max\{a^{p-1},b^{p-1}\}|a-b|,\quad a,b\ge 0,
\end{equation}
by \eqref{eq:2.2},  one has
$$
\|u(t)-v(t)\|_{L^\infty} \le  C \int_0^t \|u(s)-v(s)\|_{L^\infty} \,ds,\quad 0\le t\le T.
$$
This together with the Gronwall inequality imply that $u(t,x)= v(t,x)$ for  all $t\in[0,T]$ and  $x\in\mathbb R^N$.

For the existence of solutions,
given $0<T\le 1$, we define the set 
\begin{equation}
\label{eq:2.4}
\mathcal{V}=\left\{u\in C([0,T],C_0(\mathbb{R}^N)):\, \|u\|_{L^\infty((0,T),L^\infty)}\leq 2\delta_\infty(u_0,w)\right\},
\end{equation}
where
$\delta_\infty(u_0,w)=\max\left\{\|u_0\|_{L^\infty},\|w\|_{L^\infty}\right\}$. We endow $\mathcal{V}$ with the distance generated by the norm of $C([0,T],C_0(\mathbb{R}^N))$, that is,
\begin{equation}\label{eq:2.5}
d(u,v)=\|u-v\|_{L^\infty((0,T),L^\infty)},\quad u,v\in \mathcal{V}.
\end{equation}
Given $u\in \mathcal{V}$, let 
\begin{equation}
\label{eq:2.6}
(Fu)(t):=e^{t\Delta} u_0+\int_0^t e^{(t-s)\Delta} \left(|u(s)|^p+s^\sigma w\right)\,ds,\quad 0\leq t\leq T.
\end{equation}
Since $u_0,w\in C_0(\mathbb R^N)$, $\sigma>-1$ and $u\in\mathcal V$,
one can easily verify that $Fu\in C([0,T],C_0(\mathbb{R}^N))$.  
On the other hand,
by \eqref{eq:2.2}, one has
\begin{align*}
\|(Fu)(t)\|_{L^\infty} 
&
\le \|e^{t\Delta} u_0\|_{L^\infty}
+\int_0^t \left\|e^{(t-s)\Delta}|u(s)|^p\right\|_{L^\infty} \,ds
+\int_0^t s^\sigma \left \|e^{(t-s)\Delta} w\right \|_{L^\infty}\,ds
\\
&
\le \|u_0\|_{L^\infty}
+T \|u\|_{L^\infty((0,T),L^\infty)}^p
+\frac{T^{\sigma+1}}{\sigma+1}\|w\|_{L^\infty}
\\
&
\le 
\left(1+\frac{T^{\sigma+1}}{\sigma+1}\right)\delta_\infty(u_0,w)
+2^p\delta_\infty(u_0,w)^pT
\end{align*}
for all $0<t\le T$. This yields 
\begin{equation}
\label{eq:2.7}
\|Fu\|_{L^\infty((0,T),L^\infty)}
\le \left(1+\frac{T^{\sigma+1}}{\sigma+1}+2^p\delta_\infty(u_0,w)^{p-1}T\right)\delta_\infty(u_0,w).
\end{equation}
Let $T>0$ be a sufficiently small constant such that
\begin{equation}
\label{eq:2.8}
\frac{T^{\sigma+1}}{\sigma+1}+2^p\delta_\infty(u_0,w)^{p-1}T\le 1.
\end{equation}
Then, by \eqref{eq:2.7}, one obtains
$$
\|Fu\|_{L^\infty((0,T),L^\infty)}\le 2 \delta_\infty(u_0,w),
$$
which yields $F(\mathcal{V})\subset \mathcal{V}$. Furthermore,
for $u,v\in \mathcal{V}$,
by \eqref{eq:2.2}, \eqref{eq:2.3} and \eqref{eq:2.6}, one has
\begin{equation}
\label{eq:2.9}
\begin{split}
\|(Fu)(t)-(Fv)(t)\|_{L^\infty}
&
\le \int_0^t \left\|e^{(t-s)\Delta} \left(|u(s)|^p-|v(s)|^p\right)\right\|_{L^\infty}\,ds
\\
&
\le \int_0^t \| |u(s)|^p-|v(s)|^p\|_{L^\infty}\,ds
\\
&
\le p \delta_\infty(u_0,w)^{p-1}T\|u-v\|_{L^\infty((0,T),L^\infty)}
\end{split}
\end{equation}
for all $0<t\le T$. By \eqref{eq:2.8} and \eqref{eq:2.9}, one obtains
$$
\|Fu-Fv\|_{L^\infty((0,T),L^\infty)}\leq \frac{p}{2^p} \|u-v\|_{L^\infty((0,T),L^\infty)}.
$$
Since $2^p>p$, under the condition \eqref{eq:2.8},
the self-mapping $F: \mathcal{V}\to \mathcal{V}$ is a contraction. 
Moreover, since $(\mathcal{V},d)$ is a complete metric space,
from the Banach contraction principle, 
it follows that \eqref{eq:1.2} admits  a solution $u\in \mathcal{V}$, 
which is the unique solution to \eqref{eq:1.2} in $C([0,T],C_0(\mathbb{R}^N))$. Furthermore, since $\sigma>-1$ and  under the assumption on the function $w$,  applying similar arguments in regularity theorems for second order parabolic equation (see e.g. \cite[Chapter~1]{F}) to \eqref{eq:1.2},
one see that the solution $u$ satisfies \eqref{eq:1.1} in the classical sense, and the assertion~(i) follows.
\vspace{3pt}

Next we prove the assertion~(ii).
Applying the uniqueness of solutions, 
we see that the solution $u$, which is obtained above, 
can be extended to a maximal interval $[0,T_{\max})$, where
$$
T_{\max}=\sup\left\{t>0:\, \mbox{\eqref{eq:1.2} admits a solution in }C([0,t],C_0(\mathbb{R}^N))\right\}.
$$
Suppose that $T_{\max}<\infty$, and there exists $M>0$ such that 
\begin{equation}
\label{eq:2.10}
\|u(t)\|_{L^\infty}\le M,\quad 0\le t< T_{\max}.
\end{equation}
Let $t_*$ be such that $T_{\max}/2<t_*<T_{\max}$.
For $0<\tau<T_{\max}$, we define the set
$$
\mathcal{W}=\left\{v\in C([0,\tau],C_0(\mathbb{R}^N)):\, \|v\|_{L^\infty((0,\tau),L^\infty)}\le 2 \delta_\infty(M,w)\right\},
$$
where $\delta_\infty(M,w)=\max\left\{M,\|w\|_{L^\infty}\right\}$.
Given $v\in \mathcal{W}$, let
$$
(Gv)(t):=e^{t\Delta}u(t_*)+\int_0^t e^{(t-s)\Delta}|v(s)|^p\,ds+\int_0^t (s+t_*)^\sigma e^{(t-s)\Delta} w \,ds,\quad 0\le t\le \tau.
$$
Similarly to $\mathcal V$ and $Fu$,
we endow $\mathcal{W}$ with the distance $d$, which is defined by \eqref{eq:2.5},
and we have $Gv\in C([0,\tau],C_0(\mathbb{R}^N))$. 
Furthermore, by \eqref{eq:2.2} and \eqref{eq:2.10}, we obtain
\begin{equation}
\label{eq:2.11}
\begin{split}
\|(Gv)(t)\|_{L^\infty} 
&
\le \|e^{t\Delta} u(t_*)\|_{L^\infty}
+\int_0^t \left\|e^{(t-s)\Delta} |v(s)|^p\right\|_{L^\infty} \,ds
\\
& 
+\int_0^t (s+t_*)^\sigma \left \|e^{(t-s)\Delta} w\right \|_{L^\infty}\,ds
\\
&
\le  
\|u(t_*)\|_{L^\infty}+\tau \|v\|_{L^\infty((0,\tau),L^\infty)}^p+\frac{(t+t_*)^{\sigma+1}-t_*^{\sigma+1}}{\sigma+1}\|w\|_{L^\infty}
\\
&
\le  
\delta_\infty(M,w)+2^p\delta_\infty(M,w)^p\tau+\frac{(t+t_*)^{\sigma+1}-t_*^{\sigma+1}}{\sigma+1}\delta_\infty(M,w)
\end{split}
\end{equation}
for all $0\le t\le \tau$. On the other hand, 
applying the mean value theorem,
we see that, for any $0<t\le \tau$, there exists a constant $c_{t,t_*}\in(t_*,t+t_*)$ such that
\begin{equation}
\label{eq:2.12}
\frac{(t+t_*)^{\sigma+1}-t_*^{\sigma+1}}{\sigma+1} =c_{t,t_*}^\sigma t\le c_{t,t_*}^\sigma \tau.
\end{equation}
Since it holds from the definition of $t_*$ that
$$
\frac{T_{\max}}{2}<t_*<c_{t,t_*}<t+t_*<2T_{\max},
$$
by \eqref{eq:2.12}, we have
$$
\frac{(t+t_*)^{\sigma+1}-t_*^{\sigma+1}}{\sigma+1} \le C_\sigma \tau
$$
for all $0<t\le\tau$, where $C_\sigma=T_{\max}^\sigma\max\{2^\sigma,2^{-\sigma}\}$.
This together with \eqref{eq:2.11} implies that
\begin{equation}
\label{eq:2.13}
\|Gv\|_{L^\infty((0,\tau),L^\infty)}
\le \left(1+2^p\delta_\infty(M,w)^{p-1}\tau+C_\sigma \tau\right)\delta_\infty(M,w).
\end{equation}
Let $\tau>0$ be a sufficiently small constant such that
\begin{equation}
\label{eq:2.14}
2^p\delta_\infty(M,w)^{p-1}\tau+C_\sigma \tau\le1.
\end{equation}
Then, by \eqref{eq:2.13}, we obtain
$$
\|Gv\|_{L^\infty((0,\tau),L^\infty)}\le 2\delta_\infty(M,w),
$$
which yields $G(\mathcal{W})\subset \mathcal{W}$. Furthermore, similarly to \eqref{eq:2.9} with \eqref{eq:2.14}, one can see that  under the condition \eqref{eq:2.14}, the self-mapping $G: \mathcal{W}\to \mathcal{W}$ is a contraction. Applying the Banach contraction principle, 
we see that there exists a unique function $v\in \mathcal{W}$ satisfying
$$
v(t)=e^{t\Delta}u(t_*)+\int_0^t e^{(t-s)\Delta}|v(s)|^p\,ds+\int_0^t (s+t_*)^\sigma e^{(t-s)\Delta} w \,ds,\quad 0\le t\le \tau.
$$
For $\max\{T_{\max}/2,T_{\max}-\tau\}<\tilde t<T_{\max}$, let
$$
\widetilde{u}(t)
=\left\{
\begin{array}{lll}
u(t) &\mbox{if}& 0\le t\le \tilde t,\vspace{5pt}\\
v(t-\tilde t) &\mbox{if}& \tilde t\le t\le \tilde t+\tau.
\end{array}
\right.
$$
Then we observe that $\widetilde{u}\in C([0,\tilde t+\tau],C_0(\mathbb{R}^N))$ is a solution to \eqref{eq:1.2}
and $\tilde t+\tau>T_{\max}$, which contradicts the definition of $T_{\max}$. 
Hence, we see that if $T_{\max}<\infty$, then $\lim_{t\to T_{\max}^-}\|u(t)\|_{L^\infty}= \infty$,
and the assertion~(ii) follows.
\vspace{3pt}

Finally we prove the assertion~(iii).
Instead of the functional space $\mathcal{V}$ given by \eqref{eq:2.4}, 
we define the set
\begin{align*}
&
\mathcal{V}_r=\left\{u\in C([0,T],C_0(\mathbb{R}^N))\cap C([0,T],L^r):\right.
\\
&
\hspace{3cm}
\left. \|u\|_{L^\infty((0,T),L^\infty)}\le 2 \delta_\infty(u_0,w),\quad
\|u\|_{L^\infty((0,T),L^r)}\le 2 \delta_r(u_0,w)
\right\},
\end{align*}
where $\delta_r(u_0,w)=\max\left\{\|u_0\|_{L^r},\|w\|_{L^r}\right\}$.
We endow $\mathcal V_r$ with the distance
$$
d_r(u,v)=\|u-v\|_{L^\infty((0,T),L^\infty)}+\|u-v\|_{L^\infty((0,T),L^r)},
\qquad u,v\in \mathcal{V}_r.
$$
Since it holds that $\||u(t)|^p\|_{L^r}\le \|u(t)\|^{p-1}_{L^\infty}\|u(t)\|_{L^r}$,
applying same argument as in the proof of the assertion~(i),
we obtains a unique solution $u$ in $\mathcal{V}_r$,
and we see that $u\in C([0,T_{\max}),C_0(\mathbb{R}^N))\cap C([0,T_{\max}),L^r)$.
Thus the assertion~(iii) follows.
\end{proof} 

\section{Blow-up of solutions}\label{section:3}

In order to prove the blow-up results given by Theorems \ref{Theorem:1.1} and \ref{Theorem:1.2}, we use the well-known rescaled test function method (see \cite{MP}).
\vspace{5pt}

\begin{proof}[Proof of the assertion~(i) of Theorem \ref{Theorem:1.1}]

We argue by contradiction.
Suppose that $T_{\max}=\infty$, i.e. $u\in C([0,\infty),C_0(\mathbb{R}^N))$ is a global solution of \eqref{eq:1.2}.  We need to introduce  two cut-off functions. Let $\xi,\eta\in C^\infty([0,\infty))$ satisfy
\begin{equation}
\label{eq:3.1}
0\le \xi\le 1;\,\,\, \xi\equiv 1 \mbox{ in } [0,1];\,\,\, \xi\equiv 0 \mbox{ in } [2,\infty)
\end{equation}
and
\begin{equation}\label{etafunction}
\eta\geq 0,\,\,\, \eta\not\equiv 0,\,\,\, \mbox{supp}(\eta)\subset (0,1).
\end{equation}
For sufficiently large positive constant $T$, we put 
$$
\varphi_T(t,x)=\eta_T(t) \mu_T(x),\quad (t,x)\in [0,T]\times \mathbb{R}^N,
$$
where
\begin{equation}\label{eq:3.2}
\eta_T(t)=\eta\left(\frac{t}{T}\right)^{\frac{p}{p-1}},\quad  \mu_T(x)=\xi\left(\frac{|x|^2}{T}\right)^{\frac{2p}{p-1}}.
\end{equation}
By \eqref{eq:3.1} and \eqref{eq:3.2}, it can be easily seen that
\begin{equation}\label{eq:3.3}
|\Delta \mu_T(x)|\le \frac{C}{T}\xi\left(\frac{|x|^2}{T}\right)^{\frac{2}{p-1}},\qquad x\in\mathbb R^N.
\end{equation}
Since the solution $u$ of \eqref{eq:1.2} satisfies \eqref{eq:1.1} in the classical sense, multiplying \eqref{eq:1.1} by $\varphi=\varphi_T$,  and integrating by parts over $(0,T)\times\mathbb R^N$,
we obtain
\begin{align*}
&
\int_0^T\int_{\mathbb{R}^N} |u|^p\varphi_T \,dx\,dt
+\int_0^T\int_{\mathbb{R}^N} t^\sigma w(x) \varphi_T\,dx\,dt 
+\int_{\mathbb{R}^N} u_0(x)\varphi_T(0,x)\,dx \\
&
= -\int_0^T\int_{\mathbb{R}^N} u\Delta \varphi_T\,dx\,dt
-\int_0^T\int_{\mathbb{R}^N} u\partial_t\varphi_T\,dx\,dt.
\end{align*}
On the other hand, by \eqref{etafunction}, it holds that 
$$
 \int_{\mathbb{R}^N} u_0(x)\varphi_T(0,x)\,dx =\eta_T(0)\int_{\mathbb{R}^N}u_0(x) \mu_T(x)\,dx =0.
$$
Therefore, we deduce that 
\begin{equation}\label{eq:3.4}
\begin{split}
&
\int_0^T\int_{\mathbb{R}^N} |u|^p\varphi_T \,dx\,dt
+\int_0^T\int_{\mathbb{R}^N} t^\sigma w(x) \varphi_T\,dx\,dt 
\\
&
\le \int_0^T\int_{\mathbb{R}^N} |u| |\Delta \varphi_T|\,dx\,dt
+\int_0^T\int_{\mathbb{R}^N} |u| |\partial_t\varphi_T|\,dx\,dt.
\end{split}
\end{equation}
We claim that 
\begin{equation}
\label{eq:3.5}
\int_0^T\int_{\mathbb{R}^N} t^\sigma w(x) \varphi_T\,dx\,dt \ge C T^{\sigma+1} \int_{\mathbb{R}^N} w(x)\,dx.
\end{equation}
Indeed, we have
\begin{equation}\label{eq:3.6}
\int_0^T\int_{\mathbb{R}^N} t^\sigma w(x) \varphi_T\,dx\,dt
=\left(\int_0^T t^\sigma \eta\left(\frac{t}{T}\right)^{\frac{p}{p-1}}\,dt\right)
\left(\int_{\mathbb{R}^N} w(x)\xi\left(\frac{|x|^2}{T}\right)^{\frac{2p}{p-1}}\,dx\right).
\end{equation}
From the conditions imposed on the function $w$, and  by the dominated convergence theorem,  we obtain
$$
\lim_{T\to \infty} \int_{\mathbb{R}^N} w(x) 
\xi\left(\frac{|x|^2}{T}\right)^{\frac{2p}{p-1}}\,dx
=\int_{\mathbb{R}^N} w(x) \,dx>0.
$$
This implies  that, for a sufficiently large $T>0$, we have
\begin{equation}\label{eq:3.7}
\int_{\mathbb{R}^N} w(x) 
\xi\left(\frac{|x|^2}{T}\right)^{\frac{2p}{p-1}}\,dx\geq \frac{1}{2} \int_{\mathbb{R}^N} w(x) \,dx.
\end{equation}
On the other hand, we have
\begin{equation}\label{eq:3.8}
\int_0^T t^\sigma \eta\left(\frac{t}{T}\right)^{\frac{p}{p-1}}\,dt
=T^{\sigma+1} \int_0^1 s^\sigma \eta(s)^{\frac{p}{p-1}}\,ds.
\end{equation}
Using \eqref{etafunction}, \eqref{eq:3.6}, \eqref{eq:3.7} and \eqref{eq:3.8},  \eqref{eq:3.5} follows. 

Next, applying the $\varepsilon$-Young inequality with $\varepsilon=\frac{1}{2}$, 
we obtain
\begin{equation}
\label{eq:3.9}
\int_0^T\int_{\mathbb{R}^N} |u| |\Delta \varphi_T|\,dx\,dt
\le \frac{1}{2} \int_0^T\int_{\mathbb{R}^N} |u|^p\varphi_T \,dx\,dt+C I_1(T)
\end{equation}
and
\begin{equation}
\label{eq:3.10}
\int_0^T\int_{\mathbb{R}^N} |u| |\partial_t \varphi_T|\,dx\,dt
\le \frac{1}{2} \int_0^T\int_{\mathbb{R}^N} |u|^p\varphi_T \,dx\,dt+CI_2(T), 
\end{equation}
where
$$
I_1(T):=\int_0^T\int_{\mathbb{R}^N} \varphi_T^{\frac{-1}{p-1}}|\Delta \varphi_T|^{\frac{p}{p-1}} \,dx\,dt,
\qquad
I_2(T):=\int_0^T\int_{\mathbb{R}^N} \varphi_T^{\frac{-1}{p-1}}|\partial_t\varphi_T|^{\frac{p}{p-1}} \,dx\,dt.
$$
By the definition of $\varphi_T$ with \eqref{eq:3.2} and \eqref{eq:3.3}, we obtain
\begin{equation*}
\begin{split}
&
I_1(T)= \left(\int_0^T \eta_T(t)\,dt\right) \left(\int_{\mathbb{R}^N} \mu_T(x)^{\frac{-1}{p-1}} |\Delta \mu_T(x)|^{\frac{p}{p-1}}\,dx\right)\\
&
\le 
CT \int_{\mathbb{R}^N} \mu_T(x)^{\frac{-1}{p-1}} |\Delta \mu_T(x)|^{\frac{p}{p-1}}\,dx
\le
C T^{1+\frac{N}{2}-\frac{p}{p-1}} \int_{1<|y|< \sqrt{2}} 1\,dy,
\end{split}
\end{equation*}
which yields
\begin{equation}\label{eq:3.11}
I_1(T)\leq C T^{1+\frac{N}{2}-\frac{p}{p-1}}.
\end{equation}
Similarly, we have
\begin{equation}\label{eq:3.12}
I_2(T)
= \left(\int_0^T \eta_T(t)^{\frac{-1}{p-1}} |\eta_T'(t)|^{\frac{p}{p-1}}\,dt\right) \left(\int_{\mathbb{R}^N} \mu_T(x)\,dx\right).
\end{equation}
On the other hand,  we have
\begin{equation}\label{eq:intmu}
\int_{\mathbb{R}^N} \mu_T(x)\,dx=T^{\frac{N}{2}} \int_{|y|<\sqrt 2} \xi(|y|^2)^{\frac{2p}{p-1}} \,dy=C T^{\frac{N}{2}} 
\end{equation}
and
\begin{equation}\label{eq:inteta}
\int_0^T \eta_T(t)^{\frac{-1}{p-1}} |\eta_T'(t)|^{\frac{p}{p-1}}\,dt=\lambda^{\frac{p}{p-1}} T^{1-\frac{p}{p-1}} \int_0^1 |\eta'(s)|^{\frac{p}{p-1}}\,ds.
\end{equation}
Therefore, using \eqref{eq:3.12}, \eqref{eq:intmu} and \eqref{eq:inteta}, we obtain
\begin{equation}\label{eq:3.12bis}
I_2(T)\leq C T^{1+\frac{N}{2}-\frac{p}{p-1}}.
\end{equation}
Hence, combining \eqref{eq:3.4}, \eqref{eq:3.5}, \eqref{eq:3.9}, \eqref{eq:3.10}, \eqref{eq:3.11} and \eqref{eq:3.12bis}, we see that
$$
T^{\sigma+1}\int_{\mathbb{R}^N} w(x) \,dx
\le I_1(T)+I_2(T)
\le CT^{1+\frac{N}{2}-\frac{p}{p-1}},
$$
which yields
\begin{equation}
\label{eq:3.13}
\int_{\mathbb{R}^N} w(x) \,dx\le C T^{\frac{N}{2}-\sigma-\frac{p}{p-1}}.
\end{equation}
Passing to the limit as $T\to \infty$ in \eqref{eq:3.13} with \eqref{eq:1.3}, we obtain
$$
\int_{\mathbb{R}^N} w(x) \,dx\le 0,
$$
which contradicts the fact that $\int_{\mathbb{R}^N} w(x) \,dx>0$. 
This completes the proof of the assertion~(i) of Theorem \ref{Theorem:1.1}.
\end{proof}

\begin{proof}[Proof of Theorem \ref{Theorem:1.2}]
As previously, suppose that $u\in C([0,\infty),C_0(\mathbb{R}^N))$ is a global solution of \eqref{eq:1.2}.
For sufficiently large positive constants $T$ and $R$, we put 
$$
\psi_{T,R}(t,x)=\eta_T(t) \mu_R(x),\quad (t,x)\in [0,T]\times \mathbb{R}^N,
$$
where $\eta_T$ is given by \eqref{eq:3.2}, 
$$
\mu_R(x)=\xi\left(\frac{|x|^2}{R^2}\right)^{\frac{2p}{p-1}},\quad x\in \mathbb{R}^N
$$
and 
$\xi\in C^\infty([0,\infty))$ is  a cut-off function satisfying \eqref{eq:3.1}. Replacing $\varphi_T$ with $\psi_{T,R}$
and
applying same arguments as in the proof of the assertion~(i) of Theorem~\ref{Theorem:1.1}, we obatin
\begin{equation}
\label{eq:3.14}
\int_{\mathbb{R}^N} w(x) \,dx
\le C \left(T^{-\sigma}R^{N-\frac{2p}{p-1}}+T^{-\frac{p}{p-1}-\sigma}R^N\right).
\end{equation}
Fixing $R$ and passing to the limit as $T\to \infty$ in \eqref{eq:3.14}, since $\sigma>0$, we obtain
$$
\int_{\mathbb{R}^N} w(x) \,dx\le 0,
$$
which contradicts the fact that $\int_{\mathbb{R}^N} w(x) \,dx>0$,
and the proof of Theorem \ref{Theorem:1.2} is complete.
\end{proof}

\section{Global existence}\label{section:4}

The following lemma will be used in the proof of the global existence part of Theorem \ref{Theorem:1.1}.

\begin{lemma}\label{Lemma:4.1}
Let $N\geq 2$ and $-1<\sigma<0$. Assume \eqref{eq:1.4}.
Then 
\begin{equation}
\label{eq:4.1}
2\sigma p^2-(N+2\sigma-2)p+N<0.
\end{equation}
\end{lemma}

\begin{proof}
Assume \eqref{eq:1.4}. Then
$$
N\ge s^*:=2\sigma+\frac{2p}{p-1}.
$$
Consider the function
$$
\varrho(s)=2\sigma p^2-(s+2\sigma-2)p+s,\quad s\ge s^*.
$$
Since it follows from $p>1$ that $\varrho$ is a decreasing function, 
we obtain $\varrho(s)\le \varrho(s^*)$ for $s\ge s^*$.
On the other hand, we have $\varrho(s^*)=2 \sigma (p-1)^2<0$.
Therefore we see that $\varrho(s)<0$ for $s\ge s^*$,
and taking $s=N$ in this inequality, \eqref{eq:4.1} follows.
\end{proof}

\begin{proof}[Proof of the assertion~(ii) of Theorem \ref{Theorem:1.1}]
The proof is inspired by that of \cite[Theorem 1.1]{Caz}. By \eqref{eq:1.4} and \eqref{eq:4.1}, we can take a positive constant $q$ satisfying
\begin{equation}
\label{eq:4.2}
\frac{2}{N}\max\left\{\frac{1}{p(p-1)},\sigma+\frac{1}{p-1}\right\}<\frac{1}{q}<\min\left\{\frac{2}{N(p-1)},\frac{1}{p}\right\}.
\end{equation}
Furthermore, it follows that 
\begin{equation}
\label{eq:4.3}
q>d>k\ge 1, 
\end{equation}
where $d$ and $k$ are given by \eqref{eq:1.5}. Let 
$$
\beta=\frac{1}{p-1}-\frac{N}{2q}.
$$
Then we verify easily that 
\begin{equation}
\label{eq:4.4}
\beta >0,\qquad\beta p<1
\end{equation}
and
\begin{equation}\label{eq:4.5}
\beta-\frac{N}{2}\left(\frac{1}{d}-\frac{1}{q}\right)
=\beta(1-p)+1 -\frac{N}{2q}(p-1)
=\beta-\frac{N}{2}\left(\frac{1}{k}-\frac{1}{q}\right)+\sigma+1=0.
\end{equation}
Let $\delta$ be a sufficiently small positive constant.
We define the set
$$
\Xi=\left\{u\in L^\infty((0,\infty),L^q(\mathbb{R}^N)):\, \sup_{t>0} t^\beta \|u(t)\|_{L^q}\le \delta\right\}.
$$
We endow $\Xi$ with the distance
$$
d(u,v)=\sup_{t>0} t^\beta \|u(t)-v(t)\|_{L^q},\quad u,v\in \Xi.
$$
Then $(\Xi,d)$ is a complete metric space. 
Given $u\in \Xi$, let 
\begin{equation}
\label{eq:4.6}
(Su)(t):=e^{t\Delta}u_0+\int_0^t e^{(t-s)\Delta} |u(s)|^p\,ds +\int_0^t s^\sigma e^{(t-s)\Delta} w\,ds,\quad t\ge 0.
\end{equation}
Since $u_0\in L^d$, 
by \eqref{eq:2.1} and \eqref{eq:4.5} we have
\begin{equation}
\label{eq:4.7}
\|e^{t\Delta}u_0\|_{L^q}
\le c_1 t^{-\frac{N}{2}\left(\frac{1}{d}-\frac{1}{q}\right)} \|u_0\|_{L^d}
=c_1t^{-\beta} \|u_0\|_{L^d},\quad t>0,
\end{equation}
where $c_1$ is the constant given in \eqref{eq:2.1}. Furthermore, since it follows form \eqref{eq:4.2} that $q>p$,
by \eqref{eq:1.4} and \eqref{eq:4.5}, we obtain
\begin{equation}
\label{eq:4.8}
\begin{split}
\int_0^t \left\|e^{(t-s)\Delta} |u(s)|^p\right\|_{L^q}\,ds 
&\le c_1 \int_0^t (t-s)^{-\frac{N}{2q}(p-1)} \| |u(s)|^p\|_{L^{\frac{q}{p}}}\,ds 
\\
&
\le c_1\delta^p \int_0^t s^{-\beta p}(t-s)^{-\frac{N}{2q}(p-1)}\,ds 
\\
&
= c_1\delta^p  t^{-\frac{N}{2q}(p-1)+1-\beta p} B\left(1-\beta p,1-\frac{N}{2q}(p-1)\right)
\\
&
= c_1C \delta^p  t^{-\beta},\quad t> 0,
\end{split}
\end{equation}
where $B$ denotes the beta function. We note that by \eqref{eq:4.2} and \eqref{eq:4.4}, $B\left(1-\beta p,1-\frac{N}{2q}(p-1)\right)$ is well-defined.  Similarly, it holds that
\begin{equation}
\label{eq:4.9}
\begin{split}
\int_0^t s^\sigma \left\|e^{(t-s)\Delta} w\right\|_{L^q}\,ds 
&
\le  c_1 \|w\|_{L^k}\int_0^t s^\sigma (t-s)^{-\frac{N}{2}\left(\frac{1}{k}-\frac{1}{q}\right)} \,ds 
\\
&
= c_1 t^{-\frac{N}{2}\left(\frac{1}{k}-\frac{1}{q}\right)+1+\sigma} 
B\left(\sigma+1, 1-\frac{N}{2}\left(\frac{1}{k}-\frac{1}{q}\right)\right) \|w\|_{L^k}
\\
&
= c_1C t^{-\beta}\|w\|_{L^k},\quad t>0.
\end{split}
\end{equation}
We note again that by $\sigma>-1$ and \eqref{eq:4.2}, 
$B\left(\sigma+1, 1-\frac{N}{2}\left(\frac{1}{k}-\frac{1}{q}\right)\right)$ is well-defined. Then, by \eqref{eq:4.6}, \eqref{eq:4.7}, \eqref{eq:4.8} and \eqref{eq:4.9}, we have
$$
t^\beta \|(Su)(t)\|_{L^q} \le C_*\left(\|u_0\|_{L^d}+\delta^p+\|w\|_{L^k}\right),\quad t>0,
$$
where $C_*>0$ is a constant, independent of $\delta$. Therefore, we can chose a sufficiently small positive constant $\delta$ satisfying
$$
0<\delta\le \left(\frac{1}{2C_*}\right)^{\frac{1}{p-1}},
$$
and if the initial data $u_0$ and the inhomogeneous term $w$ satisfy
$$
\|u_0\|_{L^d}+\|w\|_{L^k}\le \frac{\delta}{2C_*},
$$
we get
$$
C_*\left(\|u_0\|_{L^d}+\delta^p+\|w\|_{L^k}\right)\le \delta.
$$
This yields $S(\Xi)\subset \Xi$. Furthermore, assuming $\|u_0\|_{L^d}+\|w\|_{L^k}$ and $\delta$ small enough if necessary, and applying similar arguments as above, we see that the self-mapping $S: \Xi\to \Xi$ is a contraction, so it admits a fixed point $u\in L^\infty((0,\infty),L^q)$, which solves \eqref{eq:1.2}. We claim that 
\begin{equation}
\label{eq:4.10}
u\in C([0,\infty),C_0(\mathbb{R}^N)).
\end{equation}
In order to prove our claim,  we first show that for $T>0$ small enough, 
$u\in C([0,T],C_0(\mathbb{R}^N))$.  For any $T>0$ (small enough),
we observe that the above argument yields uniqueness in
$$
\Xi_T=\left\{u\in L^\infty((0,T),L^q):\, \sup_{0<t<T} t^\beta \|u(t)\|_{L^q}\le \delta\right\}.
$$
Let $\widetilde{u}$ be the local solution to \eqref{eq:1.2} obtained by Proposition~\ref{Proposition:2.1}.
Since it follows from \eqref{eq:4.3} that $u_0,w\in L^q$,
by Proposition~\ref{Proposition:2.1}~(iii)
we have $\widetilde{u}\in C([0,T_{\max}),C_0(\mathbb{R}^N))\cap C([0,T_{\max}),L^q)$. Then, by the boundedness of $\|\tilde u(t)\|_{L^q}$,
for a sufficiently small $T>0$,
we see that $\sup_{0<t<T} t^\beta \|\widetilde{u}(t)\|_{L^q}\le \delta$. 
Hence, by the uniqueness of solutions, we deduce that $u=\widetilde{u}$ in $[0,T]$, so that 
\begin{equation}
\label{eq:4.11}
u\in C([0,T],C_0(\mathbb{R}^N)).
\end{equation}
Next, applying a bootstrap argument, we show that $u\in C([T,\infty),C_0(\mathbb{R}^N))$. Indeed, for $t>T$, it holds that
\begin{align*}
u(t)-e^{t\Delta}u_0-\int_0^t s^\sigma e^{(t-s)\Delta}w\,ds
&=
\int_0^T  e^{(t-s)\Delta} |u(s)|^p\,ds +\int_T^t  e^{(t-s)\Delta} |u(s)|^p\,ds
\\
&:= J_1(t)+J_2(t).
\end{align*}
Since $u\in C([0,T],C_0(\mathbb{R}^N))$, 
we can easily show that $J_1\in C([T,\infty),C_0(\mathbb{R}^N))$.
Furthermore, by the above calculations used to construct the fixed point,
we have $J_1\in C([T,\infty),L^q)$. On the other hand,
by \eqref{eq:4.2}, we see that $q>N(p-1)/2$,
and we can take a constant $r\in(q,\infty]$ such that
$$
\frac{N}{2}\left(\frac{p}{q}-\frac{1}{r}\right)<1.
$$
Since $u\in L^\infty((0,\infty),L^q)$,
for $\widetilde{T}>T$, 
we know that  $|u|^p\in L^\infty((T,\widetilde{T}),L^{\frac{q}{p}})$, 
and it easily follows that $J_2\in C([T,\widetilde{T}],L^r)$. By the arbitrariness of $\widetilde{T}$, it  holds $J_2\in C([T,\infty),L^r)$.
Since the terms $e^{t\Delta}u_0$, $\int_0^t s^\sigma e^{(t-s)\Delta}w\,ds$ 
and $J_1$ belong to $C([T,\infty),C_0(\mathbb{R}^N))\cap C([T,\infty),L^q)$,
we deduce that $u\in C([T,\infty), L^r)$. Iterating this process a finite number of times, we obtain
\begin{equation}
\label{eq:4.12}
u\in C([T,\infty),C_0(\mathbb{R}^N)).
\end{equation}
Hence, \eqref{eq:4.10} follows from \eqref{eq:4.11} and \eqref{eq:4.12},
and the proof of the assertion~(ii) of Theorem \ref{Theorem:1.1} is complete.
\end{proof}

\vspace{1cm}
\noindent {\bf Acknowledgements.} 
The first author is supported by Researchers Supporting Project
number (RSP-2019/57), King Saud University, Riyadh, Saudi Arabia.
The second author  was also supported in part by the Grant-in-Aid for Young Scientists (B)
(No.~16K17629) and the Grant-in-Aid for Scientific Research (S)(No.~19H05599)
from Japan Society for the Promotion of Science.

\vskip 3mm

\font\tenrm=cmr10
{\tenrm
\noindent Mohamed Jleli,
Department of Mathematics, College of Science, King Saud University, P.O. Box 2455, Riyadh, 11451, Saudi Arabia. 
E-mail: jleli@ksu.edu.sa
\vskip 1.5mm

\noindent Tatsuki Kawakami,
 Department of Applied Mathematics and Informatics, Ryukoku University, Seta, Otsu, Japan.
E-mail: kawakami@math.ryukoku.ac.jp
\vskip 1.5mm

\noindent Bessem Samet,
Department of Mathematics, College of Science, King Saud University, P.O. Box 2455, Riyadh, 11451, Saudi Arabia. 
E-mail: bsamet@ksu.edu.sa

}

\end{document}